\documentclass[11pt]{amsart}
\usepackage{amsfonts,amscd,latexsym,amsmath,amssymb,enumerate,verbatim,amsthm,epsfig,alltt,enumitem,color,tikz-cd}

\newcommand{\Rea}{\mathbb{R}}
\newcommand{\Nat}{\mathbb{N}}
\newcommand{\Int}{\mathbb{Z}}

\newcommand{\NN}{\mathcal{N}}

\newcommand{\DD}{\mathcal{D}}

\newcommand{\Img}{\mathrm{Img}}

\theoremstyle{plain}

\newtheorem{theorem}{\bf Theorem}[section]
\newtheorem{lemma}[theorem]{\bf Lemma}
\newtheorem{proposition}[theorem]{\bf Proposition}
\newtheorem{corollary}[theorem]{\bf Corollary}
\newtheorem{fact}[theorem]{\bf Fact}

\theoremstyle{definition}
\newtheorem{definition}[theorem]{\bf Definition}

\title[Shadowing of actions of hyperbolic groups]{Shadowing of actions of hyperbolic groups on their boundaries}
\author{Michal Doucha}
\address{Institute of Mathematics\\
	Czech Academy of Sciences\\
	\v Zitn\'a 25\\
	115 67 Praha 1\\
	Czech Republic}
\email{doucha@math.cas.cz}
\urladdr{https://users.math.cas.cz/~doucha/}

\keywords{hyperbolic groups, shadowing, pseudo-orbit tracing property, Gromov boundary, topological stability}
\subjclass[2020]{20F67, 37B65, 37B25, 37B05}
\thanks{The author was supported by the GA\v{C}R project 22-07833K and by the Czech Academy of Sciences (RVO 67985840).}
\begin{document}
\maketitle
\begin{abstract}
We prove that the canonical action of every hyperbolic group on its Gromov boundary has the shadowing (aka pseudo-orbit tracing) property. In particular, this recovers the results of Mann et al. (\cite{BoMa,MaMa23,MaMaWe}) that such actions are topologically stable.
\end{abstract}
\section{Introduction}

Shadowing, also known as the pseudo-orbit tracing property, is a fundamental dynamical notion having its origins in smooth dynamics in the study of hyperbolic systems. An $\varepsilon$-pseudo-orbit of a homeomorphism $f:X\to X$ on a compact metric space $X$, for some $\varepsilon>0$, is a sequence $(x_n)_{n\in\Int}\subseteq X$ such that for every $n\in\Int$, $f(x_n)$ is in an $\varepsilon$-neighborhood of $x_{n+1}$. A homeomorphism has the shadowing if each pseudo-orbit stays close to a true orbit; typical examples are uniformly hyperbolic dynamical systems. Maps with the shadowing find practical applications when computing numerically the orbits as these numerical approximations usually start rapidly diverging from the true orbit, however there is another point nearby whose true orbit is close to the approximate pseudo-orbit. Shadowing obviously has a theoretical importance as well and we refer to the monograph \cite{Palm} for a comprehensive treatment. The theory of shadowing has been eventually generalized to actions of more general groups (see \cite{OsTi})) and there is no increased conceptual difficulty in such generalizations. Indeed, the natural setting for systems with the shadowing is that of general countable acting groups. Osipov and Tikhomirov in \cite{OsTi} show that an action of a finitely generated nilpotent group has the shadowing if and only if at least one element of the group acts expansively with the shadowing. Meyerovitch shows in \cite{Mey} that a large class of algebraic actions of countable groups has the shadowing, in particular every expansive principal algebraic action of any countable group has the shadowing. In symbolic dynamics, Chung and Lee show that a subshift over a countable group has the shadowing if and only if it is of finite type (see \cite{ChKeo}) generalizing the result of Walters for the integer actions (see \cite{Walt}). In fact, subshifts of finite type play an important role in the general theory of shadowing. Indeed, generalizing the fundamental results of Good and Meddaugh for the integer actions from \cite{GoMe}, Lin, Chen and Zhou show that every action of a finitely generated group on a zero-dimensional compact metrizable space has the shadowing if and only if it is an inverse limit of an inverse system of subshifts of finite type with the Mittag-Leffler property (see \cite{LiChZh}). The author of this article shows in \cite{Do24} that a generic action of a finitely generated group $G$ on the Cantor space has the shadowing if and only if there is a comeager conjugacy class in the space of actions of $G$ on $2^\Nat$ which is if and only if the so-called projectively isolated subshifts (see \cite[Definition 2.22]{Do24}) over $G$ are dense in the spaces of subshifts over $G$.

There is also a connection between shadowing and topological stability. Walters shows in \cite{Walt} that every expansive homeomorphism with the shadowing is topologically stable and this is generalized to general group actions by Chung and Lee in \cite{ChKeo}.

More particular examples of group actions with the shadowing are still needed though and here we show that a certain fundamental class of group actions consists of systems with the shadowing. We recall that for each hyperbolic group $G$, its Gromov boundary $\partial G$ is a natural object associated to it, which is a compact Hausdorff space, and the group $G$ is inseparable from the canonical action $G\curvearrowright \partial G$. Indeed, many of the properties of hyperbolic groups themselves are naturally investigated using the dynamics of $G\curvearrowright\partial G$ (see e.g. \cite[Part 2]{Can}). Moreover, a fundamental result of Bowditch is that the special dynamics of $G\curvearrowright\partial G$ in fact characterizes hyperbolic groups, see \cite{Bow98}. The following is the main result of this paper.

\begin{theorem}\label{thm:main}
	Let $G$ be a hyperbolic group. Then the canonical action $G\curvearrowright \partial G$ of the group on its Gromov boundary has the shadowing (pseudo-orbit tracing property).
\end{theorem}

Recently, in a series of papers, Bowden and Mann \cite{BoMa}, Mann and Manning \cite{MaMa23}, and in full generality Mann, Manning and Weisman \cite{MaMaWe} show that the canonical action of a hyperbolic group on its Gromov boundary is topologically stable.

Their result follows as a corollary of Theorem~\ref{thm:main}, \cite{ChKeo}, and the fact that canonical actions of hyperbolic groups on their boundaries are expansive \cite[Chapter 3]{CoPabook}.
\begin{corollary}[\cite{MaMaWe}]
Let $G$ be a hyperbolic group. Then the canonical action $G\curvearrowright \partial G$ of the group on its Gromov boundary is topologically stable.
\end{corollary}
We remark that very recently Mann, Manning and Weisman also proved a version of topological stability for relatively hyperbolic groups and their canonical actions on Bowditch boundaries in \cite{MaMaWe24}.\medskip

In Section~\ref{sec:prelim} we recall the notions we work with in this paper. Section~\ref{sec:boundaries} contains basic results on Gromov boundaries of hyperbolic groups that will be needed in the proof of the main theorem which is carried out in the final Section~\ref{sec:mainproof}.
\section{Preliminaries}\label{sec:prelim}
\subsection{Metric geometry}
A \emph{geodesic}, resp. \emph{geodesic ray} in a metric space $M$ is an isometric embedding $c:[a,b]\to M$, resp. $c:\left[0,\infty\right)\to M$, for $a<b\in\Rea$. Sometimes, where there is no risk of confusion, we also call by geodesic, resp. geodesic ray, the images of the respective maps. Otherwise, we write $\Img(c)$ for the image of the geodesic, resp. geodesic ray $c$.

Since the metric spaces we shall work with are mostly finitely generated groups with the word metrics which are integer valued, we shall call by geodesic, resp. geodesic ray also the isometric embeddings $c:[n,m]\to G$, resp. $c:\Nat\to G$, where $G$ is a finitely generated group with some word metric, $[n,m]$ is an integer-interval for some $n<m\in\Int$.

\begin{definition}[Rips-hyperbolic metric spaces]
Let $M$ be a geodesic metric space and let $\delta>0$. $M$ is called \emph{$\delta$-hyperbolic} (in the sense of Rips) if for any $x,y,z\in M$ and any geodesics $c_{x,y}$, $c_{y,z}$, and $c_{z,x}$ between the corresponding points, for any $c\in\{c_{x,y}, c_{y,z},c_{z,x}\}$ we have that $\Img(c)$ is contained in the $\delta$-neighborhood of $\Img(c')\cup\Img(c'')$, where $c'\neq c''\in \{c_{x,y}, c_{y,z},c_{z,x}\}\setminus c$.

It is (Rips)-hyperbolic if it is $\delta$-hyperbolic for some $\delta\in\left[0,\infty\right)$.
\end{definition}

Let $G$ be a finitely generated group and let $S\subseteq G$ be some finite symmetric generating set. The \emph{word metric on $G$ with respect to $S$} is the metric $d_S$ defined by \[d_S(g,h):=\min\{n\in\Nat\colon \exists s_1,\ldots,s_n\in S\;(g=hs_1\cdots s_n)\},\quad g\neq h\in G.\] It is well-known and easy to check that for a different choice $S'\subseteq G$ of a finite symmetric generating set the word metrics $d_S$ and $d_{S'}$ are bi-Lipschitz equivalent.

\begin{definition}
A finitely generated group $G$ is (Gromov)-hyperbolic if $G$ is hyperbolic with respect to some, equivalently any, word metric on $G$. Equivalently, $G$ is hyperbolic if it acts geometrically on a hyperbolic metric space by isometries.
\end{definition}

\begin{definition}[Gromov boundary of a hyperbolic space]
Let $(M,d_M)$ be a hyperbolic proper geodesic metric space. Let $\mathcal{R}$ be the set of all geodesic rays $c:\left[0,\infty\right)\to M$, or rather their images, equipped with the topology of uniform convergence on bounded subsets. Let $D_\mathcal{R}$ denote the Hausdorff metric induced by $d_M$ and let $\sim$ be the equivalence relation on $\mathcal{R}$, where for $c,c'\in\mathcal{R}$ we have $c\sim c'$ if $D_\mathcal{R}(c,c')<\infty$. Denote by $\partial M$ the space $\mathcal{R}/\sim$ with the quotient topology. It is called the \emph{Gromov boundary of $M$}.
\end{definition}

When two geodesic rays have finite Hausdorff distance we shall say they are \emph{asymptotic}.

\subsection{Dynamics of groups}
The following is the central dynamical notion in this paper. We provide two equivalent definitions. We shall use the former in this article, the latter is more common in the literature.
\begin{definition}
	Let $G$ be a countable group acting continuously on a compact metrizable space $X$. We say that the action has \emph{the shadowing} (or \emph{the pseudo-orbit tracing property}) if one, and so both, of these two equivalent conditions hold.
	\begin{enumerate}
		\item  For every finite open cover $(U_i)_{i=1}^n$ of $X$ there exist a finite open cover $(V_j)_{j=1}^m$ of $X$ and a finite set $S\subseteq G$ such that for every sequence $(x_g)_{g\in G}\subseteq X$ satisfying for all $g\in G$ and $s\in S$ that $sx_g$ and $x_{sg}$ lie in some $V_j$, $j\leq m$, there exists $x\in X$ such that for all $g\in G$, $gx$ and $x_g$ lie in some $U_i$, $i\leq n$.
		
		To shorten the notation, we shall call such a sequence $(x_g)_{g\in G}$ an $(S,(V_j))$-pseudo-orbit.\medskip
		
		\item Suppose $X$ is equipped with a compatible metric $d_X$. Then for every $\varepsilon>0$ there are $\gamma>0$ and a finite set $S\subseteq G$ such that for every sequence $(x_g)_{g\in G}\subseteq X$ satisfying for all $g\in G$ and $s\in S$ that $d_X(sx_g,x_{sg})<\gamma$, there exists $x\in M$ such that for all $g\in G$ we have $d_X(gx,x_g)<\varepsilon$.
	\end{enumerate}

Notice that the first definition is applicable even without the condition on $X$ being metrizable.
	 
\end{definition}

We remark that when $G$ is finitely generated, by some finite generating set $T\subseteq G$, then the finite set $S\subseteq G$ from the definition can be always taken to be the fixed generating set $T$. Although we shall work only with finitely generated groups it is still (mostly notationally) useful to work with the more flexible definition where $S$ is allowed to vary.
\section{Boundaries and quasi-geodesics}\label{sec:boundaries}
In this section we collect several results that will be needed in the proof of the main result. Most of them are variants of well-known facts or can be quickly derived from such.

From now on, a hyperbolic group $G$ generated by a finite symmetric set $S\subseteq G$ is fixed. We have that the space $(G,d_S)$ is $\delta$-hyperbolic for some $\delta\in\Nat\cup\{0\}$.

\subsection{Geometry of quasi-geodesics}

We recall that a quasi-geodesic in a metric space $M$ is a quasi-isometric embedding of an interval (bounded or unbounded) into $M$. Since we work with a very special metric space we shall use the following more restrictive definition.
\begin{definition}
Let $\lambda\in\left(0,1\right]$ and $\varepsilon\in\Nat_0$. A \emph{$(\lambda,\varepsilon)$-quasi-geodesic} is a map $c:N\to G$ from a (bounded or unbounded) discrete interval $N\subseteq \Int$ satisfying \[\forall s,t\in N\; \big(\lambda|s-t|-\varepsilon\leq d_S(c(s),c(t))\leq |s-t|\big).\]

A map $c:N\to G$, where again $N\subseteq \Int$ is a bounded or unbounded interval is a \emph{$k$-local $(\lambda,\varepsilon)$-quasi-geodesic}, for some $k\in\Nat$, if the restriction of $c$ to every subinterval of length at most $k$ is a $(\lambda,\varepsilon)$-quasi-geodesic, i.e. \[\forall s,t\in N\; \big(|s-t|\leq k\Rightarrow \lambda|s-t|-\varepsilon\leq d_S(c(s),c(t))\leq |s-t|\big).\]
\end{definition}

\begin{lemma}\label{lem:localquasigeo}
Let $\lambda\in\left(0,1\right]$ and $\varepsilon\in\Nat_0$. Then there is $k_0>8\delta$ such that for every $k\geq k_0$ there are $\lambda'\in\left(0,1\right]$ and $\varepsilon'\in\Nat_0$ such that every $k$-local $(\lambda,\varepsilon)$-quasi-geodesic $c:N\to G$ is a $(\lambda',\varepsilon')$-quasi-geodesic.
\end{lemma}
\begin{proof}
It follows from \cite[Theorem III.H 1.7 and Theorem III.H 1.13]{BriHaebook}, see also the remark on page 407 in \cite{BriHaebook}.
\end{proof}

The following is the well-known Morse lemma whose proof can be found e.g. in \cite[Theorem III.H 1.7]{BriHaebook} or \cite[Theorem 11.40]{DruKabook}.
\begin{lemma}\label{lem:Morse}
For every $\lambda\in\left(0,1\right]$ and $\varepsilon\in\Nat_0$ there exists $K\geq 0$ such that for every $(\lambda,\varepsilon)$-quasi-geodesic $c:N\to G$, where $N\subseteq\Nat$ is a bounded or unbounded interval containing $1$ (i.e. either equal to $\Nat$ or equal to $[1,n]$ for some $n\in\Nat$) such that
\begin{enumerate}
	\item if $N$ is finite, then every geodesic $c'$ with the same endpoints as $c$ is of Hausdorff distance at most $K$ from $c$;
	\item if $N$ is infinite, then there exists a geodesic $c':\Nat\to G$ such that $c(1)=c'(1)$ and $c$ and $c'$ are of Hausdorff distance at most $K$.
\end{enumerate} 
\end{lemma}

\begin{proposition}\label{prop:closenessofgeo}
	Let $c,c':[1,R]\to G$, for $R\geq 1$ be two geodesics. Set $D:=d_S(c(1),c'(1))$ and $T:=d_S(c(R),c'(R))$. Then for every $t\leq R-T-2\delta$ we have $d_S(c(t),c'(t))\leq D+4\delta$.
	
	Consequently, for every two asymptotic geodesic rays $c,c':\Nat\to G$ and for every $t\in\Nat$ we have $d_S(c(t),c'(t))\leq d_S(c(1),c'(1))+4\delta$.
\end{proposition}
\begin{proof}
	Fix three geodesics $v$, $v'$, $v''$ resp. between $c(1)$ and $c'(1)$, between $c(R)$ and $c'(R)$, and between $c(1)$ and $c'(R)$ respectively. Pick $t\leq R-T-2\delta$. Then $c(t)$ is of distance at most $\delta$ from either $v'$ or $v''$. Since $t\leq R-T-2\delta$, a simple triangle inequality shows it must be of distance at most $\delta$ from some element $x\in \Img(v'')$. Then $x$ is of distance at most $\delta$ from either $v$ or $c'$. 
	\begin{enumerate}
		\item Suppose there is $t'\leq D$ such that $d_S(x,v(t'))\leq \delta$. By triangle inequality, we have $t\leq 2\delta+t'\leq D+2\delta$. Then we have \[\begin{split}d_S(c(t),c'(t)) &\leq 2\delta+d_S(v(t'),c'(1))+t\\ &\leq 2\delta+(D-t+2\delta)+t=D+4\delta.\end{split}\]\medskip
		
		\item Suppose now that there is $t'$ such that $d_S(x,c'(t'))\leq \delta$. Either $t'\leq t$ or $t'\geq t$. In the first case, by triangle inequality, $t\leq D+t'+2\delta$, while in the second case, by triangle inequality, $t'\leq D+t+2\delta$. Therefore $|t-t'|\leq D+2\delta$ and so by triangle inequality \[d_S(c(t),c'(t))\leq 2\delta+|t-t'|\leq D+4\delta.\]
	\end{enumerate}
	
	The `Consequently' part immediately follows.
\end{proof}

\begin{definition}
	Let $c_1,c_2:N\to G$, where $N\subseteq\Nat$ is a bounded, resp. unbounded interval containing $1$ (i.e. either equal to $\Nat$ or equal to $[1,n]$ for some $n\in\Nat$), be two geodesics, resp. geodesic rays such that $c_1(1)=c_2(1)$. Call such maps \emph{geodesics stemming from the same point}, resp. \emph{geodesic rays stemming from the same point}. We say that $c_1$ and $c_2$ \emph{$(l,D)$-fellow travel}, for some $l\in\Nat$, $D\geq 2\delta$, if for every $i\leq l$ we have $d_S(c_1(i),c_2(i))\leq D$.
\end{definition}

\begin{lemma}\label{lem:closenessofgeo}
	Let $c,c':\Nat\to G$ be geodesic rays stemming from the same point. Suppose that for some  $K\in\Nat$ and $t>K+\delta$ we have $d_S(c(t),\Img(c'))\leq K$. Then for every $t'\leq t-K-\delta$ we have $d_S(c(t'),c'(t'))\leq 2\delta$.
\end{lemma}
\begin{proof}
	Let $s\in \Nat$ be such that $d_S(c(t),c'(s))\leq K$ and let $c''$ be a geodesic between $c(t)$ and $c'(s)$. Consider the geodesic triangle $c\upharpoonright [1,t]$, $c''$, $c'\upharpoonright[1,s]$. For every $t'\leq t-K-\delta$, by $\delta$-hyperbolicity, there is a point lying either on $\Img(c'')$ or $\Img(c'\upharpoonright [1,s])$ of distance at most $\delta$ from $c(t')$. By triangle inequality, the point must lie on $\Img(c'\upharpoonright [1,s])$, thus we have $d_S(c(t'),\Img(c'))\leq \delta$, and by another triangle inequality, $d_S(c(t'),c'(t'))\leq 2\delta$.
\end{proof}

\begin{lemma}\label{lem:fellowtravel}
For every $l'$ and $D\geq 2\delta$ there exists $l$ such that every two geodesic rays $c_1,c_2:N\to G$ stemming from the same point that $(l,D)$-fellow travel actually $(l',2\delta)$-fellow travel.
\end{lemma}
\begin{proof}
This follows immediately from Lemma~\ref{lem:closenessofgeo}.
\end{proof}

It is a well-know fact that two geodesic rays stemming from the same point either fellow-travel or eventually diverge exponentially. We shall need a consequence of this phenomenon.

\begin{fact}\label{fact:divergence}
Let $D\in\Nat\cup\{0\}$ and $R>0$ be two constants. Then for every pair $c,c':\Nat\to G$ of two geodesic rays satisfying $d_S(c(1),c'(1))\leq D$
there are constants $\Delta\in\Nat$ and $s\in\Nat$, the former depending only on the hyperbolicity constant $\delta$ and on $D$ and the latter additionally on $R$ such that either $d_S(c(t),\Img(c'))\leq \Delta$ for all $t\in\Nat$ or there is minimal $t_0\in\Nat$ such that $d_S(c(t_0),\Img(c'))>\Delta$ and we have $d_S(c(t_0+s),\Img(c'))>R$.
\end{fact}
\begin{proof}
Fix $\delta$, $D$ and $R>0$. It immediately follows from Proposition~\ref{prop:closenessofgeo} that one can take e.g. $\Delta=D+4\delta$. Suppose there is no such uniform constant $s\in\Nat$. Then there is a sequence $(c_n,c'_n)_{n\in\Nat}$ of pairs of geodesic rays such that for every $n\in\Nat$ we have $d_S(c_n(1),c'_n(1))\leq D$ and there is $t_n\in\Nat$ such that $d_S(c_n(t_n),\Img(c'_n))>\Delta$, yet $d_S(c_n(t_n+n),\Img(c'_n))\leq R$. For large enough $n$ this contradicts Proposition~\ref{prop:closenessofgeo}.
\end{proof}

We shall need a version of the previous fact for quasi-geodesics.
\begin{proposition}\label{prop:quasi-geo-diverging}
There are constants \[\Delta_1=\Delta_1(\delta,\lambda,\varepsilon,D)\quad\text{and}\quad\Delta_2=\Delta_2(\delta,\lambda,\varepsilon,D,C),\] where $\Delta_1$ depends only on $\delta$, $\lambda$, $\varepsilon$ and $D\in\Nat$, and $\Delta_2$ additionally on $C$, such that for every pair of $(\lambda,\varepsilon)$-quasi-geodesic rays $c,c':\Nat\to G$, for some $\lambda\in\left(0,1\right]$ and $\varepsilon\in\Nat_0$, with $d_S(c(1),c'(1))\leq D$ we have the following. Either \[d_S(c(t),\Img(c'))\leq \Delta_1\quad\text{for all}\quad t\in\Nat\] or there is $t_0\in\Nat$ such that $d_S(c(t_0),\Img(c'))>\Delta_1$ and then \[d_S(c(t_0+\Delta_2),\Img(c'))>\Delta_1+C+10\delta.\]
\end{proposition}
\begin{proof}
Let $r$, resp. $r'$ be two geodesic rays stemming from $c(1)$, resp. $c'(1)$ that are of Hausdorff distance $K$ from $c$, resp. $c'$ which are guaranteed together with the constant $K\in\Nat$ depending on $\delta$, $\lambda$ and $\varepsilon$ by Lemma~\ref{lem:Morse}. Let $\Delta'\in\Nat$ be the constant from Fact~\ref{fact:divergence} depending on $\delta$ and $D=d_S(r(1),r'(1))$. Set $\Delta_1:=2K+\Delta'$. Then either $r$ and $r'$ are asymptotic in which case $c$ and $c'$ are asymptotic as well, and by triangle inequality, for every $t\in\Nat$ we have $d_S(c(t),\Img(c'))\leq K+\Delta'+K=\Delta_1$. Suppose now that there is $t\in\Nat$ such that $d_S(c(t),\Img(c'))>\Delta_1$ and let $s\in\Nat$ be such that $d_S(c(t),r(s))\leq K$. Then by triangle inequality $d_S(r(s),\Img(r'))>\Delta'$, thus by Fact~\ref{fact:divergence} there is $s'\in\Nat$ such that \[d_S(r(s+s'),\Img(r'))>\Delta_1+10\delta+C+2K\] and moreover, $s'$ depends only on $\delta$, $D$, and $C$. Then again, there is $\Delta_2\in\Nat$ such that $d_S(r(s+s'),\Img(c))=d_S(c(t+\Delta_2),r(s+s'))\leq K$. We have that $\Delta_2$ only depends on $K$, $\lambda$, $\varepsilon$, and $s'$. Indeed, for every $x,y\in\Nat$ there are $x',y'\in\Nat$ such that $d_S(cx),r(x'))\leq K$ and $d_S(r(y),c(y'))\leq K$, and the dependence of $x'$ on $x$, resp. of $y'$ on $y$ is as follows, denoting $B:=\varepsilon+K$:
\[\lambda^{-1} x-B\leq x'\leq \lambda x+B,\] resp. \[\lambda^{-1}(y-B)\leq y'\leq \lambda (y+B),\] which is proved e.g. in \cite[(11.3) and (11.4) p.375]{DruKabook}.

By triangle inequality we get \[\begin{split}d_S(c(t+\Delta_2),\Img(c'))\geq  d_S(r(s+s'),\Img(r'))-d_S(c(t+\Delta_2),r(s+s'))\\ -d(\Img(r'),\Img(c'))>\Delta_1+10\delta+C+2K-K-K=\Delta_1+C+10\delta.\end{split}\]
\end{proof}

We finish this subsection with a simple but useful lemma which shows that concatenating certain geodesic segments results in a quasi-geodesic.

\begin{lemma}\label{lem:gluingtolocalquasigeo}
	Let $c_1:[n_1,n_2]\subseteq\Nat\to G$ and $c_2:[m_1,m_2]\subseteq\Nat\to G$ be two geodesics such that there is $n_1<n<n_2$ satisfying $n_2-n>n$ and $n>8\delta$, and such that $c_1(n)=c_2(m_1)$ and for every $i\leq n$ we have $d_S(c_1(n+i),c_2(m_1+i))\leq 2\delta$. Then the map $c:[n_1,n+m_2-m_1]\to G$ defined by \[c(i):=\begin{cases}
		c_1(i) & \text{if }i\leq n\\
		c_2(i-n+m_1) & \text{if }i>n
	\end{cases}\] 
	is an $n'$-local $(1,2\delta)$-quasi geodesic for every $n'\leq n$.
\end{lemma}
\begin{proof}
	Pick $k_2>k_1\in [n_1,n+m_2-m_1]$ such that $k_2-k_1\leq n'$. If either $k_1,k_2\leq n$ or $k_1,k_2\geq n$, then by definition we have $d_S(c(k_1),c(k_2))=k_2-k_1$, so we may assume that $k_1<n$ and $k_2>n$. But then we have \[d_S(c(k_1),c(k_2)) \geq d_S(c(k_1),c_1(k_2))-d_S(c_1(k_2),c_2(k_2))=|k_1-k_2|-2\delta\]
\end{proof}

\subsection{The Gromov boundary and its presentation}
Instead of representing elements of the boundary by geodesic rays it will be more useful for us to use the following notion.
\begin{definition}
A \emph{directed system of geodesics (DSG)} is a map $m:G\to S$ such that
\begin{enumerate}
	\item for every $g\in G$ the sequence $g, gm(g),gm(g)m(gm(g)),\ldots$ is a geodesic ray, denoted as $c_m^g:\Nat\to G$;
	\item for any $g,h\in G$ the geodesic rays $c_m^g$ and $c_m^h$ are asymptotic.
\end{enumerate}
\end{definition}

It follows that each DSG defines a unique element of the Gromov boundary $\partial G$. The next proposition shows that conversely each element of the boundary is defined by some DSG.

\begin{proposition}\label{prop:existenceofDSG}
Let $c:\Nat\to G$ be a geodesic ray. Then there exists DSG $m:G\to S$ such that for every $g\in G$ the geodesic rays $c$ and $c_m^g$ are at bounded distance.
\end{proposition}

Before giving the proof we need few definitions. Let $G$ be a group generated by a finite symmetric set $S\subseteq G$. Denote by $S^*$ the set of all words over $S$ as the alphabet and for every $w\in S^*$ let $w_G\in G$ be the evaluation of this word, i.e. the group element this element represents. Say that $G$ has a \emph{distinguished geodesic structure} if there is an assignment $\lambda:G\to S^*$ such that for every $g\in G$
\begin{itemize}
	\item $\lambda(g)_G=g$ and moreover $|\lambda(g)|=d_S(g,1)$;
	\item writing $\lambda(g)=s_1\ldots s_n$ we have that $\lambda((s_1\ldots s_{n-1})_G)=s_1\ldots s_{n-1}$.
\end{itemize}
Notice that every \emph{geodesically automatic group} (see \cite[Section 5.3]{HoReRobook}) has a distinguished geodesic structure, in particular every hyperbolic group (see \cite[Section 5.5.1]{HoReRobook}).
\begin{proof}[Proof of Proposition~\ref{prop:existenceofDSG}]
Let $\lambda:G\to S^*$ be a function witnessing that $G$ has a distinguished geodesic structure. For every $g\in G\setminus\{1_G\}$ we have that $\lambda(g)$ is a non-empty word $s_1\ldots s_n$ and so we may define $l(g)=s_n^{-1}$ to be the inverse of the last letter of the word $\lambda(g)$. We extend $l$ to the whole $G$ by setting $l(1_G)=*$. Now view $l$ as an element in the Bernoulli shift $(S\cup\{*\})^G$. Denote by $\sigma:G\times (S\cup\{*\})^G\to (S\cup\{*\})^G$ the shift action and consider the sequence $\big(l_n:=\sigma(c(n))(l)\big)_{n\in\Nat}$. Let $m\in (S\cup\{*\})^G$ be a cluster point of this sequence. Notice there is no $g\in G$ such that $m(g)=*$, thus $m\in S^G$. We claim that $m$ is the DSG as required.\smallskip

First we verify that $m$ is a DSG. Notice that for every $g\in G\setminus\{1_G\}$ we have that \[g, gl(g), gl(g)l(gl(g)),\ldots,1_G\] is a geodesic connecting $g$ and $1_G$. In particular, for every $n<d_S(g,1_G)$ we have that $g,gl(g),gl(g)l(gl(g)),\ldots$, where the number of terms is $n$, is a geodesic. Similarly, for $l_h:=\sigma(h)(l)$, where $h\in G$, and every $g\in G\setminus\{h\}$ we have that 
\begin{equation}\label{eq:DSG}
g,gl_h(g),gl_h(g)l_h(gl_h(g)),\ldots, h
\end{equation} is a geodesic connecting $g$ and $h$. Notice that $h$ is the unique element $f\in G$ such that $l_h(f)=*$. Using compactness, one can then easily check that for every $g\in G$ and every $n\in\Nat$ we have that $g,gm(g),gm(g)m(gm(g)),\ldots$, where the number of terms is $n$, is a geodesic, thus $c_m^g$ is a geodesic ray.

Next we show that for every $g,h\in G$ the geodesic rays $c_m^g$ and $c_m^h$ are asymptotic. By \eqref{eq:DSG}, for every $f\in G$ we have that \[g,gl_f(g),gl_f(g)l_f(gl_f(g)),\ldots,f\] is a geodesic of length $d_S(g,f)$ connecting $g$ and $f$ which we write as a map $c_f^g:[1,d_S(g,f)]\to G$, where $c_f^g(1)=g$ and $c_f^g(d_S(g,f))=f$. Similarly, we define $c_f^h:[1,d_S(h,f)]\to G$ for $h$. By Proposition~\ref{prop:closenessofgeo} applied to $c_f^g$ and $c_f^h$ we get that for every $t\leq \min\{d_S(g,f),d_S(h,f)\}-2\delta$ we have $d_S(c_f^g(t),c_f^h(t))\leq d_S(g,h)+4\delta$. Since $c_m^g$ is the pointwise limit of $(c_{c(n_k)}^g)_{k\in\Nat}$, where $(n_k)_{k\in\Nat}$ is a sequence of integers such that $\lim_{k\to\infty}\sigma(c(n_k))(l)=m$, and similarly for $c_m^h$, we get that for every $t\in\Nat$, $d_S(c_m^g(t),c_m^h(t))\leq d_S(g,h)+4\delta$.\smallskip

It remains to check that $c$ and $c_m^{1_G}$ are asymptotic. By the paragraph above, $c_m^{1_G}$ is the pointwise limit of $(c_{c(n_k)}^{1_G})_{k\in\Nat}$. For each $k\in\Nat$, $c_{c(n_k)}^{1_G}$ is a geodesic between $1_G$ and $c(n_k)$, thus uniformly close to the geodesic segment $c\upharpoonright [1,n_k]$. It follows the limit, $c_m^{1_G}$, is uniformly close to $c$.
\end{proof}

\begin{definition}
We define a compact non-Hausdorff topology $\tau$ on the set $\DD$ of all DSGs on $G$ as the topology generated by basic open neighborhoods \[\NN_m^{l,D}:=\big\{m'\in\DD\colon \forall i\leq l\;\big(d_S(c_m^{1_G}(i),c_{m'}^{1_G}(i))\leq D\big)\big\}\] where $m\in\DD$, $l>8\delta$ and $D\geq 2\delta$.
\end{definition}
The following immediately follows from Lemma~\ref{lem:fellowtravel}.

\begin{proposition}\label{prop:fellowtravel}
For every $l>8\delta$ and $D, D'\geq 2\delta$ there exists $l'$ such that for every $m\in\DD$, $\NN_m^{l',D'}\subseteq \NN_m^{l,D}$.
\end{proposition}

The next lemma provides another definition of the neighborhoods $\NN_m^{l,D}$ that will be useful in the sequel.
\begin{lemma}\label{lem:defofnbhds}
Let $m,m'\in\DD$ and $g\in G$. Fix also $l>8\delta$ and $D\geq 2\delta$. Then $gm\in\NN_{m'}^{l,D}$ if and only if for every $i\leq l$ \[d_S(g c_m^{g^{-1}}(i),c_{m'}^{1_G}(i))\leq D.\]
\end{lemma}
\begin{proof}
By definition, $gm\in\NN_{m'}^{l,D}$ means that for every $i\leq l$ we have $d_S(c_{gm}^{1_G}(i),c_{m'}^{1_G}(i))\leq D$.  Notice that also by definition, for every $\rho\in\DD$, $g\in G$, and $t\in\Nat$ we have \[c_{g\rho}^{1_G}(t)=1_G\cdot \rho(g^{-1})\cdot \rho(g^{-1}\rho(g^{-1}))\cdots\] where the number of terms is $t$, and this is by definition equal to \[g\cdot \big(g^{-1}\cdot \rho(g^{-1})\cdot\rho(g^{-1}\rho(g^{-1}))\cdots\big)=g c_\rho^{g^{-1}}(t)\] where the number of terms in the parenthesis on the left-hand side is $t$. This shows that for every $i\leq n$ we have $d_S(c_{gm}^{1_G}(i),c_{m'}^{1_G}(i))\leq D$ if and only if \[d_S(g c_m^{g^{-1}}(i),c_{m'}^{1_G}(i))\leq D.\]
\end{proof}

Let us denote by $Q:\DD\to\partial G$ the canonical surjective map.

\begin{definition}
	For every $x\in\partial G$ and $l\in\Nat$ we define the set \[\NN_x^l:=\{y\in\partial G\colon \forall m,m'\in\DD\; (Q(m)=x\wedge Q(m')=y\Rightarrow m'\in\NN_m^{l,2\delta})\}.\]
\end{definition}
It is straightforward that the sets $\NN_x^l$, $x\in\partial G$, $l\in\Nat$, form a basis of topology of $\partial G$.

One can consider also the pointwise convergence topology on $\DD$. This is a compact Hausdorff topology which is strictly finer than $\tau$; in particular every pointwise convergence sequence $(m_n)_{n\in\Nat}\subseteq\DD$ also converges in $\tau$.

\begin{proposition}
The Hausdorff quotient of $(\DD,\tau)$ is homeomorphic to $\partial G$.
\end{proposition}
\begin{proof}
Denote the Hausdorff quotient of $(\DD,\tau)$ by $X$. By the universal property of $X$, if the surjective map $Q:\DD\to \partial G$ is continuous when $\DD$ is equipped with $\tau$, then we have the following commutative diagram
\[
\begin{tikzcd}
	(\DD,\tau)  \arrow{d}{f} \arrow{dr}{Q}\\
	X \arrow{r}{\overline{Q}}  & \partial G
\end{tikzcd}
\]
where $f:(\DD,\tau)\to X$ is the quotient map and $\bar Q$ is the unique continuous map making the diagram commute. It is then clear that $\bar Q$ is a bijection and as a continuous bijection between two compact Hausdorff space it is a homeomorphism. Let $(\rho_n)_n\subseteq\DD$ be a sequence converging to $\rho\in\DD$ in $\tau$. Let $(\rho_{l_n})_n$ be an arbitrary subsequence of $(\rho_n)_n$. By properness of $(G,d_S)$, using a simple diagonalization, there is a further subsequence $(\rho_{k_n})_n$ of $(\rho_{l_n})_n$ that converges pointwise to some $\rho'\in\DD$. Since $(\rho_{k_n})_n\to_\tau \rho$ we get that for all $m\in\Nat$, $d(c_{\rho'}^{1_G}(m),c_\rho^{1_G}(m))\leq 2\delta$. Therefore $c_{\rho'}^{1_G}$ and $c_\rho^{1_G}$ are asymptotic and so $Q(\rho')=Q(\rho)$. Since the map $Q$ is, by definition of the topology of $\partial G$, continuous with respect to the pointwise convergence, we get $\lim_{n\to\infty} Q(\rho_{k_n})=Q(\rho')=Q(\rho)$. Since the subsequence $(\rho_{l_n})_n$ was arbitary we get $\lim_{n\to\infty} Q(\rho_n)=Q(\rho)$. This finishes the proof.
\end{proof}

The following simple lemma is analogous to the basic fact that each open cover in a compact metric space has a positive Lebesgue number.

\begin{lemma}\label{lem:Lebesgue}
Let $(U_i)_{i=1}^n$ be an open cover of $\partial G$. Then there exists $l\in\Nat$ such that for every $m,m'\in\DD$, if we have $m'\in\NN_m^{l,2\delta}$ then there is $i\leq n$ so that $Q(m),Q(m')\in U_i$.
\end{lemma}
\begin{proof}
Suppose that no such $l$ exists. Then there is a sequence $(m_l,m'_l)_{l\in\Nat}\subseteq \DD^2$ such that for every $l\in\Nat$ and $i\leq l$ we have $d_S(c_{m_l}^{1_G}(i),c_{m'_l}^{1_G}(i))\leq 2\delta$, yet $Q(m_l)$ and $Q(m'_l)$ do not lie in a common element of the open cover $(U_i)_{i=1}^n$. Passing to a subsequence if necesary, we may assume that both $(m_l)_l$ and $(m'_l)_l$ pointwise, and thus also in $\tau$, converge to some $m\in\DD$, resp. $m'\in\DD$. Since then $c_m^{1_G}$ and $c_{m'}^{1_G}$ are asymptotic, we have $Q(m)=Q(m')\in U_i$, for some $i\leq n$. Then by continuity $Q(m_l),Q(m'_l)\in U_i$ for large enough $l$, a contradiction.
\end{proof}

\begin{lemma}\label{lem:inverseLebesgue}
For every $l>8\delta$ there exists a finite open cover $(U_i)_{i=1}^n$ of $\partial G$ such that for every $i\leq n$ and $x,x'\in U_i$ we have $x'\in \NN_x^l$.
\end{lemma}
\begin{proof}
By Lemma~\ref{lem:fellowtravel} there exists $l'\in\Nat$ such that every two geodesic rays stemming from the same point that $(l',4\delta)$-fellow travel actually $(l,2\delta)$-fellow travel. By compactness, the open cover $(\NN_x^{l'})_{x\in\partial G}$ has a finite subcover $(U_i)_{i=1}^n$. We claim it is as desired. Pick $i\leq n$, $x,x'\in U_i$, and any $m,m'\in\DD$ such that $Q(m)=x$, $Q(m')=x'$. Then by definition and triangle inequality $m'\in\NN_m^{l',4\delta}$, thus $m'\in\NN_m^l$. It follows that $x'\in\NN_x^l$.
\end{proof}

\section{Shadowing of $G\curvearrowright \partial G$}\label{sec:mainproof}

This section is devoted to the proof of Theorem~\ref{thm:main}. The idea of the proof is that given a pseudo-orbit $(x_g)_{g\in G}\subseteq \partial G$ one chooses geodesic rays representing the elements of $(x_g)_{g\in G}$ and starts gluing together finite segments of these rays to obtain a quasi-geodesic ray that determines the element whose orbit shadows $(x_g)_{g\in G}$. The precise proof of this idea is somewhat technical using the tools we have obtained in the previous section.\bigskip

Let $(U_i)_{i=1}^n$ be an open cover of $\partial G$. We need to find an open cover $(V_j)_j$ of $\partial G$ and a finite set $F\subseteq G$ such that for every $(F,(V_j)_j)$-pseudo-orbit $(x_g)_{g\in G}\subseteq \partial G$ there exists $x\in\partial G$ shadowing the pseudo-orbit, i.e. satisfying that for every $g\in G$ there is $i\leq n$ so that $gx,x_g\in U_i$.

For every $i\leq n$, set $U'_i:=Q^{-1}(U_i)$. Then $(U'_i)_{i=1}^n$ is a finite open cover of $\DD$ and by Lemma~\ref{lem:Lebesgue} there exists $l\in\Nat$ such that for all $m,m'\in\DD$ satisfying $m'\in\NN_m^{l,2\delta}$ there is $i\leq n$ so that $m,m'\in U'_i$. 

We now set several constants. First, let $\lambda\in\left[0,1\right)$, $\varepsilon\in\Nat_0$ and $k\in\Nat$ be the constants provided by Lemma~\ref{lem:localquasigeo} guaranteeing that every $k$-local $(1,2\delta)$-quasi-geodesic is a $(\lambda,\varepsilon)$-quasi-geodesic. Next let $K\in\Nat$ be the constant provided by Lemma~\ref{lem:Morse} that guarantees that every $(\lambda,\varepsilon)$-quasi-geodesic ray is $K$-close to some geodesic ray. We set \[J:=\max\{k+1,K+2\delta+l+1\}.\] Then we apply Proposition~\ref{prop:quasi-geo-diverging} to obtain the constants $\Delta_1:=\Delta_1(\delta,\lambda,\varepsilon,1)$ and $\Delta_2:=\Delta_2(\delta,\lambda,\varepsilon,1,2J+1)$. Finally, we set \[L:=\max\{\Delta_1+2+2J,J+1+\Delta_2\}.\]\smallskip

By Lemma~\ref{lem:inverseLebesgue} there exists an open cover $(V_j)_{j=1}^m$ of $\partial G$ such that for every $j\leq m$ and $x,x'\in V_j$ we have that $x'\in\NN_x^L$.\medskip

We claim that $(V_j)_j$ along with the finite set $F:=\{g\in G\colon d_S(g,1_G)\leq L\}\subseteq G$ are as desired. We prove it in the rest of this section.\medskip

Assume that $(x_g)_{g\in G}\subseteq \partial G$ is an $(F,(V_j))$-pseudo-orbit. For each $g\in G$ choose $m_g\in Q^{-1}(x_g)\in\DD$.

Notice that we have that for every $g\in G$ and $f\in F$, $fm_g\in\NN_{m_{fg}}^{L,2\delta}$.\smallskip

First we need two lemmas.

\begin{lemma}\label{lem:1stpseudoorbitlemma}
For every $g,h\in G$ satisfying $d_S(g,h)\leq L$ and every $t\leq L$ we have \[d_S(h c_{m_{h^{-1}}}^{1_G}(t),g c_{m_{g^{-1}}}^{1_G}(t))\leq d_S(h,g)+6\delta.\]
\end{lemma}
\begin{proof}
Since $g^{-1}h\in F$ we have $g^{-1}h m_{h^{-1}}\in\NN_{m_{g^{-1}}}^{L,2\delta}$. By Lemma~\ref{lem:defofnbhds}, this is equivalent with that for every $t\leq L$  \[d_S(g^{-1}h c_{m_{h^{-1}}}^{h^{-1}g}(t),c_{m_{g^{-1}}}^{1_G}(t))\leq 2\delta\] and multiplying both sides by $g$, which is isometric, we get for every $t\leq L$
 \[d_S(h c_{m_{h^{-1}}}^{h^{-1}g}(t),g c_{m_{g^{-1}}}^{1_G}(t))\leq 2\delta.\] Moreover, by Proposition~\ref{prop:closenessofgeo}, since $c_{m_{h^{-1}}}^{1_G}$ and $c_{m_{h^{-1}}}^{h^{-1}g}$ are asymptotic and $d_S(c_{m_{h^{-1}}}^{1_G}(1),c_{m_{h^{-1}}}^{h^{-1}g}(1))=d_S(1_G,h^{-1}g)$ we get that for every $t\leq L$ \[d_S(h c_{m_{h^{-1}}}^{1_G}(t),h c_{m_{h^{-1}}}^{h^{-1}g}(t))=d_S(c_{m_{h^{-1}}}^{1_G}(t),c_{m_{h^{-1}}}^{h^{-1}g}(t))\leq d_S(1_G,h^{-1}g)+4\delta.\] The lemma follows from the triangle inequality.
\end{proof}

\begin{lemma}\label{lem:2ndpseudoorbitlemma}
For every $g\in G$, $t\leq J+\Delta_2$ and $i\leq\Delta_2$ we have, after setting $f:=g c_{m_{g^{-1}}}^{1_G}(t)$, \[d_S(g c_{m_{g^{-1}}}^{1_G}(t+i),f c_{m_{f^{-1}}}^{1_G}(i))\leq 2\delta.\]
\end{lemma}
\begin{proof}
We have $d_S(f,g)=t\leq J+\Delta_2\leq L$, so $f^{-1}g\in F$ and by the assumption we have $f^{-1}g m_{g^{-1}}\in\NN_{m_{f^{-1}}}^{L,2\delta}$. By Lemma~\ref{lem:defofnbhds}, this is equivalent that for every $i\leq L$ we have \[d_S(f^{-1}g c_{m_{g^{-1}}}^{g^{-1}f}(i),c_{m_{f^{-1}}}^{1_G}(i))\leq 2\delta\] and so by multiplying both sides by $f$, which is isometric, we get \[d_S(g c_{m_{g^{-1}}}^{g^{-1}f}(i), f c_{m_{f^{-1}}}^{1_G}(i))\leq 2\delta.\]

Since for all $n\in\Nat$ we have \[c_{m_{g^{-1}}}^{g^{-1}f}(n)=c_{m_{g^{-1}}}^{1_G}(t+n)\] the statement of the lemma follows.
\end{proof}

\noindent\textbf{Construction.} For every $g\in G$ we recursively define a map $c(g):\Nat\to G$ as follows.\medskip

\begin{enumerate}
	\item {\bf First step.} We set $c(g)(i):=c_{m_g}^{1_G}(i)$, for $i\leq J$.
	\item {\bf Second step.} set $h_2:=c(g)(J)\in F$ and for $J\leq i\leq 2J$ set \[c(g)(i):=h_2 c_{m_{h_2^{-1}g}}^{1_G}(i-J+1).\]
	\item {\bf Third general step.} Suppose that $c(g)$ has been defined on $[1,(r-1)J]$, for some $r\geq 3$, and $h_{r-1}\in F^{r-1}$ has been defined. Set $h_r:=c(g)((r-1)J)\in F^r$ and for $(r-1)J\leq i\leq rJ$ set \[c(g)(i):=h_r c_{m_{h_r^{-1}g}}^{1_G}(i-(r-1)J+1).\]
\end{enumerate}
Notice first that $c(g)(J)$ is well-defined and also for every $r\geq 1$ \[h_r c_{m_{h_r^{-1}g}}^{1_G}(J+1)=h_{r+1}c_{m_{h_{r+1}^{-1}g}}^{1_G}(1)\] thus $c(g)(rJ)$ is well-defined for every $r\in\Nat$.\medskip

\noindent{\bf Claim.} \textit{We have that $c(g)$ restricted to $[1,2J]$ and to  $[rJ,(r+2)J]$, for every $r\in\Nat$, is a $J'$-local $(1,2\delta)$-quasi geodesic, for every $k\leq J'\leq J$.}\smallskip

We only show the general case when $c(g)$ is restricted to $[rJ,(r+2)J]$, for $r\in \Nat$. The special case of the restriction to $[1,2J]$ is proved verbatim the same, the only difference is that the length of the interval $[1,2J]$ differs by one from the length of $[rJ,(r+2)J]$, for $r\in\Nat$.

Fix $r\in\Nat$. We already know that $h_r c_{m_{h_r^{-1}g}}^{1_G}(J+1)=h_{r+1} c_{m_{h_{r+1}^{-1}g}}^{1_G}(1)$ so it is enough to show that \[d_S\big(h_rc_{m_{h_r^{-1}g}}^{1_G}(J+t),h_{r+1} c_{m_{h_{r+1}^{-1}g}}^{1_G}(t)\big)\leq 2\delta\] for $t\leq k+1$, since then we are in position to apply Lemma~\ref{lem:gluingtolocalquasigeo} to reach the conclusion.

By the assumption and definition, we have $h_{r+1}^{-1}h_r m_{h_r^{-1}}g\in \NN_{m_{h_{r+1}^{-1}g}}^{L,2\delta}$, where $L> k$, since $d_S(h_{r+1}^{-1}h_r,1_G)=J\leq L$. By Lemma~\ref{lem:defofnbhds}, equivalently for every $t\leq L$ \[d_S\big(h_{r+1}^{-1}h_r c_{m_{h_r^{-1}g}}^{h_r^{-1}h_{r+1}}(t),c_{m_{h_{r+1}^{-1} g}}^{1_G}(t)\big)\leq 2\delta.\] Next we multiply by $h_{r+1}$ to get for every $t\leq L$ \[d_S\big(h_r c_{m_{h_r^{-1}g}}^{h_r^{-1}h_{r+1}}(t),h_{r+1} c_{m_{h_{r+1}^{-1} g}}^{1_G}(t)\big)\leq 2\delta.\] Finally, we realize, as in the proof of Lemma 4.2, that for every $t\leq L$ we have $c_{m_{h_r^{-1}g}}^{h_r^{-1}h_{r+1}}(t)=c_{m_{h_r^{-1}g}}^{1_G}(J+t)$, since $d_S(1_G,h_r^{-1}h_{r+1})=J$, and we get the desired \[d_S\big(h_rc_{m_{h_r^{-1}g}}^{1_G}(J+t),h_{r+1} c_{m_{h_{r+1}^{-1}g}}^{1_G}(t)\big)\leq 2\delta\]\smallskip

It follows that the whole $c(g)$ is a $k$-local $(1,2\delta)$-quasi geodesic. By Lemma~\ref{lem:localquasigeo}, $c(g)$ is a $(\lambda,\varepsilon)$-quasi geodesic. Applying Lemma~\ref{lem:Morse}, there exists a geodesic ray $r(g):\Nat\to G$ with $r(g)(1)=1_G$ that has Hausdorff distance less than $K$ to $c(g)$. Finally, applying Proposition~\ref{prop:existenceofDSG}, there exists $m(g)\in\DD$ such that the geodesic rays $r(g)$ and $c_{m(g)}^{1_G}$ are at bounded distance (at most $\delta$ since they stem from the same point).\smallskip

Now, since $c(g)\upharpoonright[1,J]=c_{m_g}^{1_G}\upharpoonright [1,J]$ we have \[d_S(c_{m_g}^{1_G}(J),\Img(c_{m(g)}^{1_G}))\leq K+\delta,\] thus by Lemma~\ref{lem:closenessofgeo}, \[d_S(c_{m_g}^{1_G}(t),c_{m(g)}^{1_G}(t))\leq 2\delta,\quad\text{for every }1\leq t\leq J-K-2\delta.\]
Since $J-K-2\delta\geq l$ we get that $m(g)\in\NN_{m_g}^l$ and thus there is $i\leq n$ such that $Q(m(g)), x_g\in U_i$.\medskip

\begin{proposition}\label{prop:finalpartofmainproof}
For every $g,h\in G$, $c_{gh^{-1} m(h)}^{1_G}$ and $c_{m(g)}^{1_G}$ are asymptotic.
\end{proposition}
Notice that as soon as we prove Proposition~\ref{prop:finalpartofmainproof} we are done. Indeed, then for any $g,h\in G$ we have \[x:=Q(m(1_G))=Q(g^{-1} m(g))=Q(h^{-1} m(h))\in \partial G\] is a well-defined element satisfying that for every $g\in G$ there is $i\leq n$ such that \[gx=gQ(m(1_G))=gQ(g^{-1}m(g))=Q(m(g)),x_g\in U_i.\]

\begin{proof}[Proof of Proposition~\ref{prop:finalpartofmainproof}]
Applying Zorn's lemma, let $M\subseteq G$ be a maximal connected subset containing $1_G$ such that for all $g,h\in M$, $c_{gh^{-1} m(h)}^{1_G}$ and $c_{m(g)}^{1_G}$ are asymptotic. If $M\neq G$, then there are $g\in M$ and $h\notin M$ such that $d_S(g,h)=1$. Thus it suffices to prove the claim for $g,h\in G$ satisfying $d_S(g,h)=1$. To ease the notation we shall prove it for $g=1_G$ and $h=s\in S$. The argument for general pairs having distance $1$ is the same.\medskip

Notice that it suffices to prove that the quasi-geodesic rays $c(1_G)$ and $s^{-1}c(s)$ are asymptotic. To further ease the notation, set $c:=c(1_G)$ and $c':=c(s)$. We recall that $\Delta_1$ and $\Delta_2$ be the constants provided by Proposition~\ref{prop:quasi-geo-diverging} depending on $\delta$, $d_S(c(1),c'(1))=1$, $\lambda$ and $C=2J+1$.

Either for every $t\in\Nat$ we have $d_S(c(t),\Img(s^{-1}c'))\leq \Delta_1$ and then $c$ and $s^{-1}c'$ are asymptotic, or there exists a minimal $t\in\Nat$ such that
 $d_S(c(t),\Img(s^{-1}c'))>\Delta_1$. The latter then implies, by Proposition~\ref{prop:quasi-geo-diverging}, that \[d_S(c(t+\Delta_2),\Img(s^{-1}c'))>\Delta_1+2J+1+10\delta\] which we shall show leads to a contradiction. Indeed, assume the latter. Let $v\in\Nat$ be such that $d_S(c(t),\Img(s^{-1}c'))=d_S(c(t),s^{-1}c'(v))>\Delta_1$. Since $t$ is minimal with this property, we have $d_S(c(t-1),\Img(s^{-1}c'))\leq \Delta_1$ and thus \[d_S(c(t),s^{-1}c'(v))=\Delta_1+1.\] By definition of $c$, resp. $c'$ there are $h,h'\in G$, $t',v'\in\Nat$ such that 
 \begin{itemize}
 	\item $c(t)=h c_{m_{h^{-1}}}^{1_G}(t')$ and $c'(v)=h' c_{m_{(h')^{-1}s}}^{1_G}(v')$,
 	\item $t'=d_S(h,c(t))\leq J$ and $v'=d_S(h',c'(v))\leq J$.
 
 \end{itemize}

Since \[\begin{split}d_S(h,s^{-1}h')&\leq  d_S(h,c(t))+d_S(c(t),s^{-1}c'(v))+d_S(s^{-1}c'(v),s^{-1}h')\\&=  d_S(h,c(t))+d_S(c(t),s^{-1}c'(v))+d_S(c'(v),h')\\ &\leq J+\Delta_1+1+J\leq L\end{split}\]

we have by Lemma~\ref{lem:1stpseudoorbitlemma} that \[d_S(h c_{m_{h^{-1}}}^{1_G}(r),s^{-1}h' c_{m_{(h')^{-1}s}}^{1_G}(r))\leq \Delta_1+1+2J+6\delta,\; \text{for all }r\leq L.\] In particular, since $t'+\Delta_2\leq L$, we have 
\begin{equation}\label{eq:1}
d_S(h c_{m_{h^{-1}}}^{1_G}(t'+\Delta_2),\Img(s^{-1}h' c_{m_{(h')^{-1}s}}^{1_G}))\leq \Delta_1+1+2J+6\delta.	
\end{equation}

Let $v''\in\Nat$ be such that 
 \begin{multline}\label{eq:2}
d_S(h c_{m_{h^{-1}}}^{1_G}(t'+\Delta_2),\Img(s^{-1}h' c_{m_{(h')^{-1}s}}^{1_G}))=\\d_S(h c_{m_{h^{-1}}}^{1_G}(t'+\Delta_2),s^{-1}h' c_{m_{(h')^{-1}s}}^{1_G}(v'+v'')). 
 \end{multline}
 
We now estimate \[d_S(c(t+\Delta_2),h c_{m_{h^{-1}}}^{1_G}(t'+\Delta_2))\] and \[d_S(c'(v+v''),h' c_{m_{(h')^{-1}s}}^{1_G}(v'+v'')).\]

We have that either $c(t+\Delta_2)=h c_{m_{h^{-1}}}^{1_G}(t'+\Delta_2)$ or there is $i\leq \Delta_2$ such that $c(t+\Delta_2)=f c_{m_{f^{-1}}}^{1_G}(\Delta_2-i)$, where $f=c(t+i)=c_{m_{h^{-1}}}^{1_G}(t'+i)$. In the former case, we have 

\begin{equation}\label{eq:3}
d_S(c(t+\Delta_2),h c_{m_{h^{-1}}}^{1_G}(t'+\Delta_2))=0,
\end{equation}

 so we assume the latter. We have $d_S(h,f)\leq d_S(h,c(t))+d_S(c(t),c(t+i))\leq J+\Delta_2\leq L$. However, then by Lemma~\ref{lem:2ndpseudoorbitlemma} we have 

 \begin{equation}\label{eq:4}
 d_S(c(t+\Delta_2),h c_{m_{h^{-1}}}^{1_G}(t'+\Delta_2)) =d_S(f c_{m_{f^{-1}}}^{1_G}(\Delta_2-i),h c_{m_{h^{-1}}}^{1_G}(t'+\Delta_2))\leq 2\delta. 
 \end{equation}

Analogously, we verify that $d_S(c'(v+v''),h c_{m_{(h')^{-1}s}}^{1_G}(v'+v''))\leq 2\delta$, thus by \eqref{eq:1}, \eqref{eq:2}, \eqref{eq:3}, \eqref{eq:4}, and triangle inequalities we obtain \[\begin{split} d_S(c(t+\Delta_2),\Img(s^{-1}c')) &\leq d_S(c(t+\Delta_2),h c_{m_{h^{-1}}}^{1_G}(t'+\Delta_2))\\ &+d_S(h c_{m_{h^{-1}}}^{1_G}(t'+\Delta_2),s^{-1}h' c_{m_{(h')^{-1}s}}^{1_G}(v'+v''))\\ &+d_S(s^{-1}h' c_{m_{(h')^{-1}s}}^{1_G}(v'+v''),s^{-1}c'(v+v''))\\ &\leq \Delta_1+1+2J+6\delta+4\delta=\Delta_1+1+2J+10\delta,\end{split}\] which is the desired contradiction.
\end{proof}

\bibliographystyle{siam}
\bibliography{references-hyperbolic}
\end{document}